\documentclass{amsart}
\usepackage{fullpage}
\usepackage{amssymb}
\usepackage{tikz}
\usetikzlibrary{matrix,arrows,decorations.pathmorphing}
\usepackage{bm}

\edef\marginnotetextwidth{\the\textwidth}

\DeclareMathOperator{\Sym}{Sym}

\newcommand{\NS}{N\'eron-Severi}

\theoremstyle{definition}

\makeatletter
\newtheorem*{rep@theorem}{\rep@title}
\newcommand{\newreptheorem}[2]
{\newenvironment{rep#1}[1]
{\def\rep@title{#2 \ref{##1}}
\begin{rep@theorem}}
{\end{rep@theorem}}}
\makeatother

\newtheorem{lemma}{Lemma}
\newreptheorem{lemma}{Lemma}
\newtheorem{theorem}{Theorem}
\newreptheorem{theorem}{Theorem}
\newtheorem{corollary}{Corollary}
\newtheorem{definition}{Definition}
\newtheorem{conjecture}{Conjecture}
\newtheorem{question}{Question}

\DeclareMathOperator{\vol}{vol}
\DeclareMathOperator{\nef}{Nef}
\DeclareMathOperator{\psef}{Psef}

\title{Asymptotic Purity for Very General Hypersurfaces of $\boldsymbol{\mathbb{P}^n\times\mathbb{P}^n}$ of bidegree $\mathbf{(k,k)}$}
\author{Michael A. Burr}
\address{Fordham University, 441 East Fordham Road, Bronx, NY 10458, USA}
\email{mburr1@fordham.edu}

\date{\today}

\begin{document}
\begin{abstract}
For a complex irreducible projective variety, the volume function and the higher asymptotic cohomological functions have proven to be useful in understanding the positivity of divisors as well as other geometric properties of the variety.  In this paper, we study the vanishing properties of these functions on specific hypersurfaces of $\mathbb{P}^n\times\mathbb{P}^n$.  In particular, we show that very general hypersurfaces of bidegree $(k,k)$ obey a very strong vanishing property, which we define as {\em asymptotic purity}: at most one asymptotic cohomological function is nonzero for each divisor.  This provides evidence for a conjecture of Bogomolov and also suggests some general conditions for asymptotic purity.
\end{abstract}
\maketitle
\section{Introduction}
Let $X$ be an irreducible projective variety over $\mathbb{C}$.  The volume function on the real \NS\ group and its generalizations, the higher asymptotic cohomological functions \cite{Kuronya:Article:Cohomological}, have been shown to capture many of the positivity properties of divisors on $X$; in this paper, all divisors will be Cartier.  For example, a divisor is big (i.e., is in the interior of the effective cone) if and only if its volume is nonzero \cite{Lazarsfeld:BookI}, and a divisor is ample if and only if there exists a neighborhood in the real \NS\ group in which all of the higher asymptotic cohomological functions vanish \cite{deFernexetal:Article}.  We expect that the higher asymptotic cohomological functions provide additional information about the non-positive divisors on a variety.  In this paper, we explore vanishing properties of these asymptotic cohomolomogical functions.  In particular, we show that when $X$ is a very general subvariety of $\mathbb{P}^n\times\mathbb{P}^n$ of bidegree $(k,k)$, a strong vanishing property, called {\em asymptotic purity} holds: for each divisor on $X$, at most asymptotic cohomological function does not vanish.

The contributions of this paper are that we present a family of varieties which enjoy the strong vanishing property of asymptotic purity.  This is the first nontrivial collection of varieties which are asymptotically pure.  Prior to this paper, Abelian varieties, homogeneous spaces, curves, and surfaces which do not contain a negative curve were known to be asymptotically pure \cite{Kuronya:Article:Cohomological,Kuronya:Thesis,MyThesis}; in addition, it was proved in \cite{MyThesis} that the only complete and simplicial toric varieties which are asymptotically pure are products of projective spaces and their quotients, for more examples of asymptotically pure and counter examples see Section \ref{sec:AP}.  A second contribution of this paper is in the computation of the asymptotic cohomology.  The computation for this paper is somewhat unusual because we compute the asymptotic cohomology using a non-reduced fiber in a linear system of divisors.  This is unusual because non-reduced varieties  frequently appear as counterexamples to conjectures.  Here, however, the computation on this non-reduced fiber shows how well-behaved the asymptotic cohomological functions are.

\subsection{Volume and Higher Asymptotic Cohomological Functions}
The motivation for the volume is based in the Riemann-Roch problem which asks to find $h^0(X,mD):=\dim{H^0(X,\mathcal{O}_X(mD))}$ for a divisor $D$ and $m\gg0$.  When $D$ is ample, Serre's vanishing theorem gives a solution to this question: in this case, for $m\gg0$, all of the higher cohomology vanishes and $h^0(X,mD)$ is equal to the Euler characteristic $\chi(X,\mathcal{O}_X(mD))$.  However, when $h^0(X,mD)$ is studied as a function of $m$ in the case where $D$ is a fixed effective but not nef divisor, then this function can behave quite poorly, e.g., the natural graded ring $\oplus H^0(X,\mathcal{O}_X(mD))$ might not be finitely generated \cite{Lazarsfeld:BookI,Zariski:Article,Cutkosky:Zariski:Article,Cutkosky:Periodicity:Article}.  It was realized in \cite{Fujita:Article,Nakayama:Article,Tsuji:Article,Tsuji:Article:II} that it makes sense to study the growth rate of the number of independent sections into $\mathcal{O}_X(mD)$ for large multiples of $D$ and that this choice avoids many of the pathologies of $h^0(X,mD)$.  Let $X$ be of dimension $n$, then the volume is defined as:
$$\vol(X,D):=\lim_{m\rightarrow\infty}\frac{h^0(X,mD)}{(m^n/n!)}.$$
Algebraically, the volume appears in Cutkosky's work on the existence (and nonexistence) of Zariski decompositions, and, analytically, the volume appears in Demailly's holomorphic Morse inequalities \cite{Demailly:HolomorphicRecent,Demailly:HolomoprhicLectures,Demailly:Original,Siu:Holomorphic,Bouche:Article}.  The volume has found several uses, e.g., it is used to bound the size of the birational automorphism group for varieties of general type in \cite{Hacon:Birational}, and, in \cite{Takayama:Article}, the volume and its related invariants has proven to be useful for additional geometric results.

The volume admits many nice properties which are summarized in \cite{Lazarsfeld:BookI,Lazarsfeld:BookII}.  We highlight some of the most interesting properties here: The volume is homogeneous and extends to a continuous function on the real \NS\ group $NS(X)\otimes\mathbb{R}$, \cite{Lazarsfeld:BookI,Boucksom:Volume}.  When $D$ is ample, the volume $\vol(X,D)$ can be represented geometrically as
$$\vol(X,D)=\int_X c_1(\mathcal{O}_X(D))^n.$$
In other words, the volume is equal to the self-intersection number of $D$ and can be computed as the volume of an associated K\"ahler metric on $X$.  In addition, the volume of a divisor $D$ is nonzero precisely when $D$ is big \cite{Lazarsfeld:BookI}.  It is also an increasing function in the ample direction, i.e., if $A$ is an ample divisor, then $\vol(X,D+tA)$ is an increasing function in $t$.

Recently, the higher dimensional analogues of the volume were introduced and studied by K\"uronya \cite{Kuronya:Article:Cohomological,Kuronya:Thesis}.  These functions were defined as follows:
$$\widehat{h}^i(X,D):=\limsup_{m\rightarrow\infty}\frac{h^i(X,mD)}{(m^n/n!)}.$$
For $i=0$, $\widehat{h}^0(X,D)$ is precisely the volume since the limsup is actually a limit; we hope that the limsup is a limit which is indeed true in many cases, but is not known in full generality\footnote{The paper \cite{Demailly:HolomorphicRecent} contains an, as of yet, unproven claim that the limsup is actually a limit and the author is aware of this issue \cite{DemaillyPC}}, see \cite{SurveyLazarsfeldetal, Hering:Toric,Kuronya:Article:Cohomological} for special cases.  Demailly, via the holomorphic Morse inequalities, has provided bounds on the asymptotic cohomological functions in terms of the integral of an associated (1,1)-form over a submanifold of the variety $X$.  These bounds can also be adapted to become the algebraic Morse inequalities, which bound the asymptotic cohomological functions in terms of appropriate intersection numbers.

K\"uronya showed that these functions are homogeneous, extend to continuous functions on the real \NS\ group, and that they enjoy several formal properties with simple statements:  For example, there is an asymptotic version of Serre's duality theorem: for any divisor $D$, $\widehat{h}^i(X,D)=\widehat{h}^{n-i}(X,D)$ \cite{Kuronya:Article:Cohomological}.  K\"uronya also proves that the asymptotic cohomological functions are well-behaved under pullbacks: if $f:Y\rightarrow X$ is a proper, surjective, and generically finite map of degree $d$ of irreducible projective varieties with $D$ a divisor on $X$, then
$$\widehat{h}^i(Y,f^\ast D)=d\cdot \widehat{h}^i(X,D).$$  Also, via the standard K\"unneth formula, K\"uronya shows that
$$\widehat{h}^i(X_1\times X_2,\pi_1^\ast D_1\otimes\pi_2^\ast D_2)\leq\binom{n_1+n_2}{n_1}\sum_{j+k=i}\widehat{h}^j(X_1,D_1)\widehat{h}^k(X_2,D_2)$$
for two varieties $X_1,X_2$ with dimensions $n_1,n_2$, projection maps $\pi_1,\pi_2$, and divisors $D_1,D_2$, respectively.  Finally, an interesting use of the higher asymptotic cohomological functions is to provide evidence when a divisor is not ample.  In particular, in \cite{deFernexetal:Article}, the authors show that a divisor $D$ is ample if and only if there is a neighborhood of that divisor in the real \NS\ group where $\widehat{h}^i$ vanishes for all $i>0$.

Even though the asymptotic cohomological functions enjoy several pleasant properties, there are some subtleties which are not well understood: For example, there are only a few classes of varieties for which the values of the asymptotic cohomological functions can be calculated or even determined to be nonzero.  In addition, it is not known, except in a few special cases when they are piecewise polynomial functions.  In this paper, we study some of the vanishing properties of these asymptotic cohomological functions.  In particular, we show that for any divisor on a very general subvariety of $\mathbb{P}^n\times\mathbb{P}^n$ of bidegree $(k,k)$, {\em at most} one asymptotic cohomological function does not vanish.
\subsection{Vanishing Theorems}
There is a vast collection of cohomological vanishing theorems for invertible sheaves over projective varieties, see \cite{Lazarsfeld:BookI,Iskovskikh:Vanishing:Survey,Kollar:Survey} for surveys.  Typically, these vanishing theorems apply only to ample or big and nef divisors and are used to show that all of the higher cohomology groups vanish.  Two exceptions to this pattern appear in a vanishing theorem of Andreotti-Grauert \cite{Andreotti:Grauert} and, more recently, in \cite{Kuronya:Subvarieties}.  The results in these papers give sufficient conditions for the vanishing of the higher cohomology groups even when the divisor under consideration might not be big.  These results, however, are typically not useful in the asymptotic case because they are too strong and restrictive; it is certainly true that if $h^i(X,mD)$ vanishes for $m\gg 0$, then the asymptotic cohomological functions also vanish. However, it is not necessary for $h^i(X,mD)$ to vanish for the asymptotic cohomological functions to vanish; it is only necessary that $h^i(X,mD)$ grows slow enough in $m$.  We exhibit this difference: when the vanishing theorem presented in \cite{Kuronya:Subvarieties} is applied in the cases considered in this paper, it cannot be used to show the vanishing of the asymptotic cohomological functions; whereas, in this paper, we prove the vanishing of these functions.

Compared to the absolute case, there are considerably fewer asymptotic vanishing or (equally interesting) asymptotic nonvanishing results.  The simplest is the asymptotic version of Serre's vanishing theorem which states that for all nef divisors, all of the higher asymptotic cohomological functions vanish \cite{Kuronya:Article:Cohomological}.  As a more general case, if the stable base locus of a big divisor is $d$-dimensional, then all asymptotic cohomology in dimension above $d$ vanish \cite{Kuronya:Article:Cohomological}.  In addition, if the application of Demailly's Morse inequalities vanishes, then this bound exhibits the vanishing of an asymptotic cohomological function \cite{Demailly:Original,Totaro:vanishing,Bouche:Article}.  These bounds, however, are not strong enough to achieve the results in this paper.  For the asymptotic cohomological functions, nonvanishing results are also interesting.  For example, for all big divisors, the volume (i.e., the zero-th cohomological function) does not vanish.  Recently, Demailly has also discussed some cases when the asymptotic Morse inequalities are tight or can be used to give lower bounds on the asymptotic cohomology (and hence, evidence for the nonvanishing of the asymptotic cohomology functions) \cite{Demailly:Original,Demailly:HolomorphicRecent,Demailly:Andreotti}.  Another nonvanishing result was provided in \cite{deFernexetal:Article}, where it was shown that in all neighborhoods of divisors which are not ample, there must be a divisor with nonvanishing higher cohomology.  A nonvanishing result which considers a related question to the one considered in this paper was provided in \cite{Bogomolov:Nonvanishing} where the authors showed that $h^1(X,S^m\Omega_X^1)$ has large growth for high symmetric tensors of the canonical sheaf when the underlying variety is sufficiently singular.
\subsection{Main Results}
In this paper, we provide positive evidence for a conjecture of Bogomolov \cite{Bogomolov:Personal} concerning the structure of asymptotic cohomological functions, see Section \ref{sec:AP}.  In particular, we prove the following result:
\begin{theorem}\label{thm:main}
Let $X$ be a very general hypersurface of $\mathbb{P}^n\times\mathbb{P}^n$ of bidegree $(k,k)$ and $D$ any divisor on $X$.  Then, there exists an $i$ such that $\widehat{h}^j(X,D)=0$ for all $j\not=i$.
\end{theorem}

I would like to thank Fedor Bogomolov and Jenia Tevelev for their valuable input and ideas during the preparation of this manuscript.  In particular, Jenia Tevelev's help was instrumental in the development of a simpler case of Theorem \ref{lem:representation}.

\section{Asymptotic Cohomological Functions and Asymptotic Purity}\label{sec:AP}
Our goal in this paper is to study vanishing properties of the asymptotic cohomological functions.  In particular, we are interested in providing necessary and sufficient conditions for the vanishing of the higher asymptotic cohomological functions in the same spirit as in the relationship between the volume and big divisors.  As a first step towards that goal, we consider divisors with at most one nonvanishing asymptotic cohomological function:
\begin{definition}
Let $X$ be an irreducible projective variety and $D$ any divisor on $X$.  $D$ is said to be {\em asymptotically pure} (abbreviated AP) if there exists an index $i$ such that $\widehat{h}^j(X,D)=0$ for all $j\not=i$.  Note that $D$ is asymptotically pure if all of the asymptotic cohomological functions vanish (in this case, the choice of $i$ is not unique).  Similarly, we say that $X$ is {\em asymptotically pure} (abbreviated AP) if every (Cartier) divisor on $X$ is asymptotically pure.  We use the same notation for asymptotic purity for a divisor and a variety, so if $D$ is both a projective subvariety and a Weil divisor corresponding to a Cartier divisor, both terms could apply, but in these cases, the sense of the purity will be clear from context.
\end{definition}

We begin with some easy examples of this concept.  All nef divisors are AP because the asymptotic version of Serre's vanishing theorem implies that only $\widehat{h}^0$ may be nonzero.  All curves are asymptotically pure; this is an easy consequence of the Riemann-Roch formula.  All Abelian varieties are AP because the nonvanishing of the asymptotic cohomology is governed by Mumford's index theorem, see \cite{Kuronya:Article:Cohomological}.  It was also shown in \cite{Kuronya:Article:Cohomological} that homogeneous spaces are AP.  A surface is AP if and only if it contains no curves with negative self-intersection number; in particular, a ruled surface is AP if and only if it is semistable \cite{Lazarsfeld:BookI}.  Projective spaces with Picard number 1 are also AP since every divisor is a multiple of an ample divisor.  In \cite{Hering:Toric}, the definition of the higher asymptotic cohomological functions was extended to complete toric varieties; in this case, the only AP complete toric varieties are products of projective spaces and their quotients by finite subgroups of the action \cite{MyThesis}.  Finally, if $f:Y\rightarrow X$ is generically finite, but not finite, then $Y$ is not AP; in particular, blowups of varieties can never be AP \cite{MyThesis,Sho:Personal}.

The question we begin to address in this paper is to find necessary and sufficient conditions for a variety to be AP.  It is certainly necessary for the big and ample cones to coincide in the \NS\ group because every big divisor has a neighborhood where $\widehat{h}^0$ does not vanish and the result of \cite{deFernexetal:Article} implies that if a divisor were big but not ample, in every neighborhood of the divisor, there is another divisor with $\widehat{h}^i$ nonvanishing for some $i>0$.  This would contradict the asymptotic purity.  This conjecture is formulated as follows:

\begin{conjecture}[\cite{Bogomolov:Personal}]\label{conj}
Let $X$ be a smooth, projective $3$-fold.  Then $X$ is AP if and only if the big and ample cones are equal.
\end{conjecture}

In higher dimensions, it seems unlikely that this condition would be enough.  For higher dimensions, we formulate this conjecture as a question:

\begin{question}
Let $X$ be a smooth, projective variety.  What conditions are necessary and sufficient for $X$ to be AP?
\end{question}

In this paper, we set out to understand and develop the above conjecture; we discuss some thoughts on the question above in the Discussion Section \ref{sec:conc}.  A natural place to study this conjecture is on hypersurfaces of $\mathbb{P}^2\times\mathbb{P}^2$.  For smooth $3$-folds of this form, the real \NS\ group can be described explicitly and it is guaranteed that the big and ample cones are equal.  In this paper, we study hypersurfaces of $\mathbb{P}^n\times\mathbb{P}^n$ because smooth hypersurfaces of this product will have relatively few nonzero asymptotic cohomology classes.  We begin by recalling some facts about the real \NS\ group of $\mathbb{P}^n\times\mathbb{P}^n$ and discuss some of our notation.

The real \NS\ group of $\mathbb{P}^n\times\mathbb{P}^n$ is $\mathbb{R}^2$ and is generated by the pullbacks of the hyperplane sections on each of the projective factors.  Let $H_1$ be the pullback of the hyperplane from the first projective space and $H_2$ the pullback of the hyperplane from the second projective space.  The intersection numbers of these divisors is very simple: $H_1^iH_2^{2n-i}=0$ if $i\not=n$ and equals $1$ when $i=n$.  Since both $H_1$ and $H_2$ are pullbacks of nef divisors, they are nef and the cone they define is contained within the nef cone.  The cone they define actually is the nef cone because every divisor outside this cone has a negative intersection number with some positive combination of these two divisors.  In particular, this shows that the nef and pseudoeffective cones are equal in $\mathbb{P}^n\times\mathbb{P}^n$.  In this paper, we will use the notation $|(k,l)|$ to refer to the complete linear series $|kH_1+lH_2|$.  Our interest is in general hypersurfaces of $\mathbb{P}^n\times\mathbb{P}^n$; we begin with a discussion of their real \NS\ group.

\begin{lemma}\label{lem:restriction}
For a general hypersurface $X$ in $|(k,k)|$, $X$ is a smooth and irreducible projective variety where the real \NS\ group $N^1(X)_{\mathbb{R}}$ is $\mathbb{R}^2$.  Moreover, the big and ample cones of $X$ are equal.  In particular, let $D=a_1H_1+a_2H_2$ be any divisor on $\mathbb{P}^n\times\mathbb{P}^n$.  If both $a_1,a_2\geq 0$, then $D|_X$ is nef and $\widehat{h}^i(X,D|_X)=0$ for $i\not=0$.  Similarly, if both $a_1,a_2\leq 0$, then $D|_X$ is $-$nef and $\widehat{h}^i(X,D|_X)=0$ for $i\not=2n-1$.  Finally, if $a_1\cdot a_2<0$, i.e., their signs differ, then $\widehat{h}^i(X,D|_X)=0$ for $i\not=n,n-1$.
\end{lemma}
\begin{proof}
The fact that a general hypersurface $X$ in $|(k,k)|$ is a smooth and irreducible projective variety follows from an application of Bertini's theorem, preceded by a Segre embedding and a product of Veronese embeddings.  Since $\mathbb{P}^n\times\mathbb{P}^n$ is $2n$ dimensional and $2n\geq 4$, Lefschetz's hyperplane theorem implies that for such a smooth hypersurface, the standard restriction of the \NS\ group is an isomorphism: i.e., $\mathbb{R}^2\simeq N^1(\mathbb{P}^n\times\mathbb{P}^n)_{\mathbb{R}}\stackrel{\sim}{\rightarrow}N^1(X)_{\mathbb{R}}$.  By the discussion preceding the lemma, it follows that for $a_1,a_2\geq 0$, $D$ is nef, and, therefore, the restriction $D|_X$ is also nef.  Then, by the asymptotic version of Serre's vanishing theorem, only $\widehat{h}^0(X,D|_X)$ can be nonzero.  Similarly, for $a_1,a_2\leq 0$, $D$ is $-$nef, and, therefore, the restriction $D|_X$ is also $-$nef.  Then, by the asymptotic version of Serre's vanishing theorem and Serre's duality theorem, only $\widehat{h}^{2n-1}(X,D|_X)$ can be nonzero (note that $2n-1$ is the dimension of $X$).  Finally, for any divisor outside of the cone generated by $H_1|_X$ and $H_2|_X$, there is a positive linear combination of these divisors with negative intersection.  This implies that the big and ample cones are equal.

The final case where one of $a_1,a_2$ is positive and the other is negative can be proved most readily from the long exact sequence in cohomology coming from the following short exact sequence of sheaves:
$$
0\rightarrow\mathcal{O}_{\mathbb{P}^n\times\mathbb{P}^n}((ma_1-1)H_1+(ma_2-1)H_2) \rightarrow\mathcal{O}_{\mathbb{P}^n\times\mathbb{P}^n}(ma_1H_1+ma_2H_2) \rightarrow\mathcal{O}_{X}(mD|_X)\rightarrow 0
$$
For $m\gg0$, most of the cohomology in the corresponding long exact sequence vanish.  In particular, by using the K\"unneth formula we arrive at the following isomorphism:
$$
H^i(\mathbb{P}^n\times\mathbb{P}^n,\mathcal{O}_{\mathbb{P}^n\times\mathbb{P}^n}(ma_1H_1+ma_2H_2))\simeq\bigoplus H^j(\mathbb{P}^n,\mathcal{O}(ma_1))\otimes H^{i-j}(\mathbb{P}^n,\mathcal{O}(ma_2)),
$$
where we see that by the computation of cohomology on projective spaces, there is at most one choice of $i$ and $j$ where a term on the RHS is nonzero.  Using these simplifications, the nonzero part of the long exact sequence in cohomology is as follows:
\begin{align*}
0\rightarrow H^{n-1}(X,\mathcal{O}_{X}(mD|_X))\rightarrow H^{n}(&\mathbb{P}^n\times\mathbb{P}^n,\mathcal{O}_{\mathbb{P}^n\times\mathbb{P}^n}((ma_1-1)H_1+(ma_2-1)H_2))\\ \rightarrow &H^{n}(\mathbb{P}^n\times\mathbb{P}^n,\mathcal{O}_{\mathbb{P}^n\times\mathbb{P}^n}(ma_1H_1+ma_2H_2)) \rightarrow
H^{n}(X,\mathcal{O}_{X}(mD|_X))\rightarrow0.
\end{align*}
This proves the lemma because all of the other cohomology groups of $\mathbb{P}^n\times\mathbb{P}^n$ in the long exact sequence vanish.  This vanishing forces the other cohomology groups to be trivial and, therefore, to have no asymptotic cohomology
\end{proof}
\section{Asymptotic Purity for Very General Hypersurfaces of $\mathbb{P}^n\times\mathbb{P}^n$ of bidegree $(k,k)$}
In this section, we prove the main result of this paper:

\begin{reptheorem}{thm:main}
Let $X$ be a very general hypersurface of $\mathbb{P}^n\times\mathbb{P}^n$ of bidegree $(k,k)$, then $X$ is asymptotically pure.
\end{reptheorem}

\begin{proof}
If $n$ is 1, then any hypersurface is a curve which is known to be AP; we, therefore, assume that $n\geq 2$.  We begin by providing a sketch of the proof: the proof can be divided up into the following three steps: First, we use the description of the \NS\ group on general divisors $X$ of $\mathbb{P}^n\times\mathbb{P}^n$ from Lemma \ref{lem:restriction} and determine what must be shown for $X$ to be AP.  Second, we study the total space of the family of divisors of bidegree $(k,k)$ and exhibit the asymptotic purity of certain divisors induced from divisors on $\mathbb{P}^n\times\mathbb{P}^n$ for a special fiber $Y$ of this family.  Finally, we use upper semi-continuity to extend the asymptotic purity for certain divisors to very general hypersurfaces in $|(k,k)|$.  Throughout this proof, we use $X$ for a very general element of $|(k,k)|$ and $Y$ for the special fiber of the family.  It is important to note that $Y$, considered as a closed subscheme of $\mathbb{P}^n\times\mathbb{P}^n$, will not be reduced, and, therefore, we do not claim that $Y$ is AP.

By Lemma \ref{lem:restriction}, it will be sufficient to show that all divisors of the form $D|_X=(a_1H_1-a_2H_2)|_X$ have at most one non-vanishing asymptotic cohomology function for $a_1,a_2\geq 0$ (note the change in sign from the discussion in Lemma \ref{lem:restriction}).  For the remainder of this proof, we study the properties of a divisor $D$ of this type.

We now use the parameterization of the complete linear series $|(k,k)|$ by the projective space $P=\mathbb{P}(H^0(\mathbb{P}^n\times\mathbb{P}^n,\mathcal{O}_{\mathbb{P}^n\times\mathbb{P}^n}(kH_1+kH_2)))$.  This can be simplified via the K\"unneth formula to $P=\mathbb{P}(\Sym^k\mathbb{C}^{n+1}\otimes\Sym^k\mathbb{C}^{n+1})$.  This corresponds to polynomials of bidegree $(k,k)$ in two sets of $(n+1)$ variables. In addition, there is the total space $T\subseteq(\mathbb{P}^n\times\mathbb{P}^n)\times P$ of the family which is flat over $P$ via the projection map.  The total space can be described explicitly in the following manner: the coordinates of $\mathbb{P}(\Sym^k\mathbb{C}^{n+1}\otimes\Sym^k\mathbb{C}^{n+1})$ correspond to the coefficients of a polynomial; then, using these coordinates as indeterminates, the variety $T$ consists of the zero set of each of these polynomials.  This $T$ fits into the following commutative diagram:
\begin{center}
\begin{tikzpicture}[>=angle 90]
\matrix(a)[matrix of math nodes,
row sep=4em, column sep=3em,
text height=1.5ex, text depth=0.25ex]
{T&(\mathbb{P}^n\times\mathbb{P}^n)\times P\\
P&\mathbb{P}^n\times\mathbb{P}^n\\};
\path[right hook->](a-1-1) edge node[above]{$i$} (a-1-2);
\path[->](a-1-1) edge node[right]{$p_2\circ i$} (a-2-1);
\path[->](a-1-2) edge node[below right]{$p_2$} (a-2-1);
\path[->](a-1-2) edge node[right]{$p_1$} (a-2-2);
\end{tikzpicture}
\end{center}
Note that over a general point of $P$, where the divisor $F$ is a smooth, irreducible variety, the fiber of the projection map is just the variety $F$, but over the divisor corresponding to $f=\left(\sum_{i=0}^n x_iy_i\right)^k$, the fiber is not reduced.  Let $Y$ be the nonreduced subscheme of $\mathbb{P}^n\times\mathbb{P}^n$ over this polynomial.  We now compute the asymptotic cohomology of the divisor $D|_Y$ in indices $n$ and $n-1$ (although asymptotic cohomology is only defined for nonreduced schemes, the growth of the cohomology is well-behaved in this case).

\begin{lemma}\label{lem:specialfiber}
Let $Y$ be the special fiber and $D$ a divisor of interest, as defined above. Then, $D|_Y$ is asymptotically pure.
\end{lemma}
For the flow in the theorem, we delay the proof of this lemma until after the proof is completed.

The total space $T$ has the following property: for any invertible sheaf $\mathcal{F}$ on $\mathbb{P}^n\times\mathbb{P}^n$, its pull-back to $T$ via $i^\ast p_1^\ast$ restricted to the fiber over any point $p\in P$ is exactly the restriction of the sheaf $\mathcal{F}$ to the subscheme of $\mathbb{P}^n\times\mathbb{P}^n$ represented in the fiber over $p$.

Note that the preceding calculation does not imply that $Y$ is AP because the \NS\ group for $Y$ may be larger than $\mathbb{R}^2$.  However, since $T$ is flat over $P$, for each integer $m$ and divisor $D$, by the upper semi-continuity theorem, there exists an open set $P_{m,D}\subseteq P$ such that if $X$ is the fiber over the point $p\in P_{m,D}$, then $h^{n-1}(X,mD|_X)\leq h^{n-1}(Y,mD|_Y)$ and $h^n(X,mD|_X)\leq h^n(Y,mD|_Y)$.  The intersection of all of the $P_{m,D}$'s is a very general set such that if $X$ is in this set, $\widehat{h}^{n-1}(X,mD|_X)\leq \widehat{h}^{n-1}(Y,mD|_Y)$ and $\widehat{h}^n(X,mD|_X)\leq \widehat{h}^n(Y,mD|_Y)$.  Since $D|_Y$ is asymptotically pure, one of the terms on the RHS of these inequalities vanishes and therefore $D|_X$ is also asymptotically pure.

This choice is for one $D$; since the real \NS\ group of $\mathbb{P}^n\times\mathbb{P}^n$ is $\mathbb{R}^2$, we take a countable dense subset of the real \NS\ group, $\{D_i\}$.  By intersecting each of these $P_{m,D_i}$ for all $m$ and $i$, we have a very general set such that if $X$ is in this set, each $D_i|_X$ is asymptotically pure.  By the continuity of the asymptotic cohomological functions and the assumption that the $D_i$ are dense, it follows that $X$ is asymptotically pure.
\end{proof}

We now provide the proof of the lemma appearing above:\\[7pt]
\noindent{\em Proof of lemma.}  To prove this result, we must show that at least one of $\widehat{h}^n(Y,D|_Y)$ or $\widehat{h}^{n-1}(Y,D|_Y)$ is zero.  We can compute the cohomology of $D|_Y$ using the following exact sequence of sheaves:
$$
0\rightarrow\mathcal{O}_{\mathbb{P}^n\times\mathbb{P}^n}(mD-Y)\rightarrow\mathcal{O}_{\mathbb{P}^n\times\mathbb{P}^n}(mD)\rightarrow\mathcal{O}_Y(mD|_Y)\rightarrow 0.
$$
Expanding the long exact sequence in cohomology, the following sequence appears:
\begin{equation*}\begin{split}
0\rightarrow H^{n-1}(Y,\mathcal{O}_Y(mD|_Y))\rightarrow H^n(\mathbb{P}^n&\times\mathbb{P}^n,\mathcal{O}_{\mathbb{P}^n\times\mathbb{P}^n}(mD-Y))\\&\stackrel{\cdot f}{\rightarrow}H^n(\mathbb{P}^n\times\mathbb{P}^n,\mathcal{O}_{\mathbb{P}^n\times\mathbb{P}^n}(mD))\rightarrow H^n(Y,\mathcal{O}_Y(mD|_Y))\rightarrow 0.\end{split}
\end{equation*}
We now show that as $m\rightarrow\infty$, either the dimensions of the kernel or the cokernel of the map $f$ are small; this will then imply the result via the exactness of the sequence.  We begin by expanding the cohomology groups appearing in the middle of this exact sequence:  We  first expand the term $H^n(\mathbb{P}^n\times\mathbb{P}^n,\mathcal{O}_{\mathbb{P}^n\times\mathbb{P}^n}(mD-Y))$.  Since $Y=kH_1+kH_2$, this cohomology group is $H^n(\mathbb{P}^n\times\mathbb{P}^n,\mathcal{O}_{\mathbb{P}^n\times\mathbb{P}^n}((ma_1-k)H_1+(-ma_2-k)H_2))$ which is also, by the K\"unneth formula, isomorphic (for $m\gg0$) to $H^0(\mathbb{P}^n,\mathcal{O}_{\mathbb{P}^n}(ma_1-k))\otimes H^n(\mathbb{P}^n,\mathcal{O}_{\mathbb{P}^2}(-ma_2-k))$.  Then, Serre's duality theorem implies that this is isomorphic to $H^0(\mathbb{P}^0,\mathcal{O}_{\mathbb{P}^n}(ma_1-k))\otimes H^0(\mathbb{P}^n,\mathcal{O}_{\mathbb{P}^n}(ma_2+k-(n+1)))^\vee$.  This can then be expanded to $\Sym^{ma_1-k}\mathbb{C}^{n+1}\otimes\Sym^{ma_2+k-(n+1)}(\mathbb{C}^{n+1})^\vee$.  Similarly, $H^n(\mathbb{P}^n\times\mathbb{P}^n,\mathcal{O}_{\mathbb{P}^n\times\mathbb{P}^n}(mD))\simeq H^0(\mathbb{P}^n,\mathcal{O}_{\mathbb{P}^n}(ma_1))\otimes H^0(\mathbb{P}^n,\mathcal{O}_{\mathbb{P}^n}(ma_2-(n+1)))^\vee\simeq\Sym^{ma_1}\mathbb{C}^{n+1}\otimes\Sym^{ma_2-(n+1)}(\mathbb{C}^{n+1})^\vee$. Now, we study the following map:
$$
\Sym^{ma_1-k}\mathbb{C}^{n+1}\otimes\Sym^{ma_2+k-(n+1)}(\mathbb{C}^{n+1})^\vee\stackrel{\cdot f}{\rightarrow}\Sym^{ma_1}\mathbb{C}^{n+1}\otimes\Sym^{ma_2-(n+1)}(\mathbb{C}^{n+1})^\vee.
$$
We now apply the isomorphism between $\mathbb{C}^{n+1}$ and $(\mathbb{C}^{n+1})^\ast$ taking $y_i\leftrightarrow y_i^\ast$, let $g=(\sum_{i=0}^n x_iy_i^\ast)^k$ be the image of $f$ under this isomorphism.  Then, this map becomes:
\begin{equation}\label{eq:sym}
\Sym^{ma_1-k}\mathbb{C}^{n+1}\otimes\Sym^{ma_2+k-(n+1)}\mathbb{C}^{n+1}\stackrel{\cdot g}{\rightarrow}\Sym^{ma_1}\mathbb{C}^{n+1}\otimes\Sym^{ma_2-(n+1)}\mathbb{C}^{n+1}.
\end{equation}

We study this map via the following lemma, which is proved in the Section \ref{sec:fiber} using standard representation theory.
\begin{lemma}\label{lem:representation}
Let $f=\left(\sum_{i=0}^n x_i\otimes y_i^\ast\right)^k\in\Sym^k\mathbb{C}^{n+1}\otimes\Sym^k(\mathbb{C}^{n+1})^\vee$.  Consider the natural map from $\Sym^{ma_1-k}\mathbb{C}^{n+1}\otimes\Sym^{ma_2+k-(n+1)}\mathbb{C}^{n+1}$ to $\Sym^{ma_1}\mathbb{C}^{n+1}\otimes\Sym^{ma_2-(n+1)}\mathbb{C}^{n+1}$ given by multiplication by $f$.  When $a_1\geq a_2$, the size of the cokernel of this map is $O(m^{2n-2})$ as $m\rightarrow\infty$ and when $a_1\geq a_2$, the size of the kernel of this map is $O(m^{2n-2})$ as $m\rightarrow\infty$ with constants depending on $n$ and $k$.
\end{lemma}

Note that in particular, this result shows that when $a_1=a_2$, the growth of both the kernel and the cokernel is $O(m^{2n-2})$.  With this result in hand, we can finish the proof of this lemma: Then, by Lemma \ref{lem:representation}, it follows that the either kernel and cokernel of the multiplication by $g$ map are of size $O(m^{2n-2})$ as $m\rightarrow\infty$.  In particular, this implies that at least one of the asymptotic cohomology groups of index $i=n,n-1$ vanish for $D|_Y$.  Therefore, $D|_Y$ is asymptotically pure.
\hfill\qed\vspace{.1in}

We discuss some unusual features of this proof in the Discussion Section \ref{sec:conc}.

\subsection{Special Cases}
We begin with the $\mathbb{P}^2\times\mathbb{P}^2$ case:  In this $X$ is a 3-fold, and, by a result in \cite{MyThesis}, the proof of the result above can be simplified.  In this case, it is necessary and sufficient to only consider the behavior of the single divisor $D=(H_1-H_2)$ on $X$.  While this simplification does not change the final result, it does simplify the proof.  In addition, if $X$ is of bidegree $(1,1)$, there is complete characterization of AP varieties.

\begin{corollary}
Let $X$ be a hypersurface of $\mathbb{P}^n\times\mathbb{P}^n$ of bidegree $(1,1)$. Then $X$ is AP if and only if it is smooth.
\end{corollary}
{\em Proof sketch.}
We prove this only in the case where $n=2$ because when $n>2$, the techniques of the proof are not more difficult, but the notation is considerably more challenging.  Consider the $GL(3,\mathbb{C})$ action on one of the $\mathbb{P}^2$'s and hence on forms of bidegree $(1,1)$; this action has three orbits with representatives $x_0y_0$, $x_0y_0+x_1y_1$, and $x_0y_0+x_1y_1+x_2y_2$.  It is easy to see that the varieties corresponding to the first two are singular and the third is nonsingular.  Specializing the argument in the theorem above since the special fiber is reduced, it follows that the third variety is AP and so are all varieties in its $GL$ orbit.  For the other two varieties, we use the map (\ref{eq:sym}) with $D=(H_1-H_2)$ above and show that the cokernel of $g$ has growth greater than $C\cdot m^3$ for some constant $C$.  Since $D^3\cdot X=0$, this will immediately imply that $D|_X$ is not AP via the asymptotic Riemann-Roch formula.

For the first choice of representative ($x_0y_0$), the image of a product of monomials in $Sym^{m-1}\mathbb{C}^3\otimes Sym^{m-2}\mathbb{C}^3$ is either $0$ or linearly independent.  The elements which map to zero do not include $y_0$ in the $Sym^{m-2}\mathbb{C}^3$ part since $y_0^\ast$ acts by differentiation.  Therefore, the kernel of multiplication by $g$ is $Sym^{m-1}\mathbb{C}^3\otimes Sym^{m-2}\mathbb{C}^2$ is of size $\binom{2+(m-1)}{2}\binom{1+(m-2)}{1}=(m^3-m)/2$, which has $C\cdot m^3$ growth and therefore $V(x_0y_0)$ is not AP.

For the second choice of representative ($x_0y_0+x_1y_1$), one notes that for any product $p$ of monomials in $Sym^m\mathbb{C}^3\otimes Sym^{m-3}\mathbb{C}^3$, there are at most two products of monomials in $Sym^{m-1}\mathbb{C}^3\otimes Sym^{m-2}\mathbb{C}^3$ whose image includes the term $p$.  In particular, for any fixed $k=0,\cdots,m-2$, the term
$$\sum_{j=0}^k(-1)^j\binom{k}{j}x_0^jx_1^{k-j}\otimes x_0^{k-j}x_1^jx_2^{m-2-k}$$ is taken to zero by $x_0\otimes x_0^\ast+x_1\otimes x_1^\ast$.  Multiplying by any monomial in $Sym^{m-1-k}\mathbb{C}^3$ factors through the multiplication by $g$ and therefore, the kernel of this multiplication has size
$$\sum_{k=0}^{m-2}\binom{2+(m-1-k)}{2}$$
which again has $C\cdot m^3$ growth and therefore $V(x_0y_0+x_1y_2)$ is not AP.\qed

\section{Bounding the Cohomology on the Special Fiber}\label{sec:fiber}
In this section, we use representation theory to bound the asymptotic cohomology of $D|_Y$.  In particular, we compute the kernel and cokernel of the map appearing in Equation (\ref{eq:sym}).

\begin{replemma}{lem:representation}
Let $f=\left(\sum_{i=0}^n x_i\otimes y_i^\ast\right)^k\in\Sym^k\mathbb{C}^{n+1}\otimes\Sym^k(\mathbb{C}^{n+1})^\vee$.  Consider the natural map from $\Sym^{ma_1-k}\mathbb{C}^{n+1}\otimes\Sym^{ma_2+k-(n+1)}\mathbb{C}^{n+1}$ to $\Sym^{ma_1}\mathbb{C}^{n+1}\otimes\Sym^{ma_2-(n+1)}\mathbb{C}^{n+1}$ given by multiplication by $f$.  When $a_1\geq a_2$, the size of the cokernel of this map is $O(m^{2n-2})$ as $m\rightarrow\infty$ and when $a_1\geq a_2$, the size of the kernel of this map is $O(m^{2n-2})$ as $m\rightarrow\infty$ with constants depending on $n$ and $k$.
\begin{proof}
Consider the standard $SL(n+1,\mathbb{C})$ action on this map: for this choice of $f$, $f$ is a $SL(n+1,\mathbb{C})$-module homomorphism, and, therefore, Schur's lemma applies in this situation.  We now break these modules into irreducible components.  In general, Pieri's formula can be used to compute the tensor product of symmetric powers of the standard irreducible representation for $SL(n+1,\mathbb{C})$.  In particular,
$\Sym^a\mathbb{C}^{n+1}\otimes\Sym^b\mathbb{C}^{n+1}$ with $a\geq b$ decomposes into irreducible representations as $\bigoplus_{i=0}^{b}\Gamma_{a+b-2i,i,0,\cdots,0}$.  Here, we are using the notation appearing in \cite{Fulton:Harris}.  In our particular case, the decompositions are as follows (for $m\gg0$):

Case 1: $a_1>a_2$.  In this case, for $m\gg0$, $ma_1-k>ma_2+k-(n+1)$ and $ma_1>ma_2-(n+1)$.  In this case, the first tensor product decomposes as
\begin{equation}\label{eq:case1a}
\Sym^{ma_1-k}\mathbb{C}^{n+1}\otimes\Sym^{ma_2+k-(n+1)}\mathbb{C}^{n+1}\simeq \bigoplus_{i=0}^{ma_2+k-(n+1)}\Gamma_{m(a_1+a_2)-(n+1)-2i,i,0,\cdots,0}.\end{equation}
Meanwhile, the second product decomposes as
\begin{equation}\label{eq:case1b}
\Sym^{ma_1}\mathbb{C}^{n+1}\otimes\Sym^{ma_2-(n+1)}\mathbb{C}^{n+1}\simeq \bigoplus_{i=0}^{ma_2-(n+1)}\Gamma_{m(a_1+a_2)-(n+1)-2i,i,0,\cdots,0}.\end{equation}
The difference between these two sums is the sum
$$
\bigoplus_{i=ma_2+1-(n+1)}^{ma_2+k-(n+1)}\Gamma_{m(a_1+a_2)-(n+1)-2i,i,0,\cdots,0}.
$$
We use the formula to compute the dimension of these irreducible components in \cite[\S15.3, p.224]{Fulton:Harris} to see that the leading term of each of these $\Gamma$'s is $m^{2n-1}a_1^{n-1}a_2^{n-1}(a_1-a_2)>0$.  Since we are adding $k$ of these, this implies that the size of the difference is $O(m^{2n-1})$, and, therefore, the kernel is at least of size $O(m^{2n-1})$.

Case 2: $a_1<a_2$.   In this case, for $m\gg0$, $ma_1-k<ma_2+k-(n+1)$ and $ma_1<ma_2-(n+1)$.  In this case, the first tensor product decomposes as
\begin{equation}\label{eq:case2}
\Sym^{ma_1-k}\mathbb{C}^{n+1}\otimes\Sym^{ma_2+k-(n+1)}\mathbb{C}^{n+1}\simeq \bigoplus_{i=0}^{ma_1-k}\Gamma_{m(a_1+a_2)-(n+1)-2i,i,0,\cdots,0}.\end{equation}
Meanwhile, the second product decomposes as
$$\Sym^{ma_1}\mathbb{C}^{n+1}\otimes\Sym^{ma_2-(n+1)}\mathbb{C}^{n+1}\simeq \bigoplus_{i=0}^{ma_1}\Gamma_{m(a_1+a_2)-(n+1)-2i,i,0,\cdots,0}.$$
The difference between these two sums is the sum
$$
\bigoplus_{i=ma_1-k}^{ma_1}\Gamma_{m(a_1+a_2)-(n+1)-2i,i,0,\cdots,0}.
$$
We again use the formula to see that the leading term of each of these $\Gamma$'s is $m^{2n-1}a_1^{n-1}a_2^{n-1}(a_2-a_1)>0$.  Since we are adding $k$ of these, this implies that the size of the difference is $O(m^{2n-1})$, and, therefore, the cokernel is at least of size $O(m^{2n-1})$.

Case 3: $a_1=a_2$.  In this case, $ma_1>ma_2-(n+1)$ and therefore the decomposition of the second product is the same as Equation \ref{eq:case1b} above.  Now, the decomposition of the first product depends on the relationship between $k$ and $n$:

If $(n+1)/2\leq n+1\leq k$ then $ma_1-k<ma_2+k-(n+1)$.  Thus, the decomposition of the first product is the same as Equation \ref{eq:case1a} above.  Now, the difference between the sums is
$$\bigoplus_{i=ma_1-k}^{ma_2-(n+1)}\Gamma_{m(a_1+a_2)-(n+1)-2i,i,0,\cdots,0}.$$
We again use the formula above to see that the leading term of each of these $\Gamma$'s is $m^{2n-1}a_1^{n-1}a_2^{n-1}(a_2-a_1)=0$.  Since we are adding the dimension of $k-n$ representations, each of size $O(m^{2n-2})$, this term is small enough to ignore.

If, instead, $(n+1)/2\leq k\leq n+1$, then the decompositions remain as above, but the difference in the direct sums become
$$\bigoplus_{i=ma_1-(n+1)}^{ma_2-k}\Gamma_{m(a_1+a_2)-(n+1)-2i,i,0,\cdots,0}.$$
We again use the formula above to see that the leading term of each of these $\Gamma$'s is $m^{2n-1}a_1^{n-1}a_2^{n-1}(a_2-a_1)=0$.  Since we are adding the dimension of $k-n$ representations, each of size $O(m^{2n-2})$, this term is small enough to ignore.

Finally, if $(n+1)/2\geq k$, then $ma_1-k>ma_2+k-(n+1)$.  Thus, the decomposition of the second product is the same as Equation \ref{eq:case2} above.  Now, the difference between the sums is
$$\bigoplus_{i=ma_2-(n+1)}^{ma_2+k-(n+1)}\Gamma_{m(a_1+a_2)-(n+1)-2i,i,0,\cdots,0}.$$
We again use the formula above to see that the leading term of each of these $\Gamma$'s is $m^{2n-1}a_1^{n-1}a_2^{n-1}(a_2-a_1)=0$.  Since we are adding the dimension of $k-n$ representations, each of size $O(m^{2n-2})$, this term is small enough to ignore.

By Schur's lemma, the multiplication by $f$ map takes irreducible representations to the same type of representations.  Therefore, this implies that the kernel or cokernel contains, at least, the irreducible representations, as described above.  For the remainder of this proof, we will show that this is the entire kernel or cokernel.  Since each representation occurs with multiplicity one, we can compute the kernel or cokernel via the {\em highest weight vector} in the representations.  In the $\Gamma_{m(a_1+a_2)-(n+1)-2i,i,0,\cdots,0}$'s in the first product, the highest weight vector can be written as
$$\sum_{j=0}^i(-1)^j\binom{i}{j}x_0^{ma_1-k-j}x_1^{j}\otimes x_0^{ma_2+k-(n+1)-i+j}x_1^{i-j}.$$
We now show that the appropriate highest weight vectors do not vanish under multiplication by $f$.  It is sufficient to show that the term with $j=0$ does not vanish when $x_0^k\otimes(x_0^\ast)^k$ is applied to it since this will uniquely have the highest power of $x_0$.  Under multiplication by this element, we have
$$\frac{(ma_2+k-(n+1)-i)!}{(ma_2-(n+1)-i)!}x_0^{ma_1}\otimes x_0^{ma_2-(n+1)-i}x_1^i$$ which does not vanish for the $i$'s of interest.  Therefore, the kernels and cokernels are exactly as described above (and the remaining kernels and cokernels are zero).  This proves the lemma.
\end{proof}
\end{replemma}

\section{Discussion and Conclusion}\label{sec:conc}
In this paper, we have provided a class of examples which are AP.  This provides evidence for Conjecture \ref{conj}, but is far from a proof and the conjecture requires further study.  One portion of this proof that is particularly interesting is the choice of the special fiber where the computation occurs.  This nonreduced fiber is one of the ``worst'' fibers in the family because, in many cases, results would break down on such a nonreduced fiber; however, even there, the asymptotic cohomology vanishes.  On the other hand, the underlying variety in this case is smooth which means that this fiber might not be as difficult as expected.  This gives some hope that there may be either a simple description or that computations on other nonreduced fibers may give further insight into this problem.

For the more general question of conditions for asymptotic purity, it seems somewhat unlikely that the condition that the big and ample cones coincide would be enough to guarantee asymptotic purity.  However, it is certainly a necessary condition due to the result of \cite{deFernexetal:Article}.  Our end goal is to be able to describe the failure of asymptotic purity as a negative property because, for instance, surfaces are AP if and only if they do not contain any curves with negative self-intersection.  Some hope for a more appropriate condition comes from the recent paper \cite{Lazarsfeld:Pseudoeffective}.  In particular, if $\psef^k(X)$, the pseudoeffective cone of codimension $k$ cycles on $X$, properly contains $\nef^k(X)$, the cone of codimension $k$ cycles dual to $\psef^{n-k}(X)$, this would be a reasonable notion of negativity.  Note that in \cite{Lazarsfeld:Pseudoeffective}, the type of containment described above is impossible, but this should not be a worry because there, they study Abelian varieties, which are guaranteed to be AP.  In addition, this condition restricts to the known condition and Conjecture \ref{conj} in the two- and three-dimensional cases, respectively.

\bibliographystyle{plain}
\bibliography{VeryGeneral}

\end{document}